\documentclass[preprint,12pt]{elsarticle}

\usepackage{amsfonts}
\usepackage{amssymb}
\usepackage{rotating}
 \usepackage{epsfig}

\usepackage{amssymb}
 \usepackage{amsthm}
 \usepackage{slashbox}

\usepackage{lscape,exscale,amsthm}
\usepackage[intlimits]{amsmath}
\usepackage{rawfonts}
\usepackage{latexsym}
\usepackage[cp850]{inputenc}
\usepackage{bbm}

\newcommand{\bol}[1]{\mbox{\boldmath$#1$}}

\newcommand{\bSigma}{\bol{\Sigma}}

\newcommand{\bx}{\mathbf{X}}

\newcommand{\by}{\mathbf{Y}}

\newcommand{\bC}{\mathbf{C}}
\newcommand{\bH}{\mathbf{H}}

\newcommand{\bI}{\mathbf{I}}

\newcommand{\bn}{\boldsymbol{\nu}}

\newcommand{\bS}{\mathbf{S}}
\newcommand{\sx}{\bar{\mathbf{x}}}
\newcommand{\Tr}{\text{tr}}

\newcommand{\bTheta}{\mathbf{\Theta}}

\usepackage{stackrel}

\usepackage{float}
\numberwithin{equation}{section}
\theoremstyle{plain}
\newtheorem{theorem}{Theorem}[section]

\newtheorem{lemma}{Lemma}[section]
\newtheorem{corollary}{Corollary}[section]

\journal{Journal of Multivariate Analysis}

\begin{document}

\begin{frontmatter}
\title{On the Strong Convergence of the Optimal Linear Shrinkage Estimator for Large Dimensional Covariance Matrix$^*$}
\cortext[cor1]{Corresponding author: Arjun K. Gupta. E-mail address: gupta@bgsu.edu. Phone: (419) 372 - 2820. Fax: (419) 372 - 6092.}


  \author[a]{Taras  Bodnar}
  \author[b]{Arjun K. Gupta \corref{cor1}}
  \author[c]{Nestor Parolya}
\address[a]{Department of Mathematics, Humboldt-University of Berlin, D-10099 Berlin, Germany}
\address[b]{Department of Mathematics and Statistics, Bowling Green State University, Bowling Green, OH 43403, USA}
\address[c]{Institute of Empirical Finance (Econometrics), Leibniz University of Hannover, 30167, Germany}
%
%


\begin{abstract}
In this work we construct an optimal linear shrinkage estimator for the covariance matrix in high dimensions. The recent results from the random matrix theory allow us to find the asymptotic deterministic equivalents of the optimal shrinkage intensities and estimate them consistently. The developed distribution-free estimators obey \textit{almost surely} the smallest Frobenius loss over all linear shrinkage estimators for the covariance matrix. The case we consider includes the number of variables $p\rightarrow\infty$ and the sample size $n\rightarrow\infty$ so that $p/n\rightarrow c\in (0, +\infty)$. Additionally, we prove that the Frobenius norm of the sample covariance matrix tends almost surely to a deterministic quantity which can be consistently estimated.
\end{abstract}

\end{frontmatter}

\section{Introduction}
Nowadays, the estimation of the covariance matrix is one of the most important problems not only in statistics but also in finance, wireless communications, biology etc. The traditional estimator of the covariance matrix, i.e. its sample counterpart, seems to be a good decision only when the dimension $p$ is much smaller than the sample size $n$. This case is called the "standard asymptotics" (see, e.g., Le Cam and Yang (2000)). Here, the sample covariance matrix is proven to be an unbiased and a consistent estimator for the covariance matrix. More problems arise when $p$ is comparable to $n$, i.e. both the dimension $p$ and the sample size $n$ tend to infinity while their ratio $p/n$ tends to a positive constant $c$. It is called the "large dimensional asymptotics" or "Kolmogorov asymptotics" (see, e.g., B\"{u}hlmann and van de Geer (2011), Cai and Shen (2011)). This type of asymptotics have been exhaustively studied by Girko (1990, 1995), where it was called the "general statistical analysis". There is a great amount of research done on the asymptotic behavior of functionals of the sample covariance matrix under the large dimensional asymptotics (see, e.g., Girko and Gupta (1994, 1996a, 1996b), Bai and Silverstein (2010)).

There are some significant improvements in the case when the covariance matrix has a special structure, e.g. sparse, low rank etc. (see, Cai et al. (2011), Rohde and Tsybakov (2011), Cai and Yuan (2012), Cai and Zhou (2012), etc.). The case when the underlying random process obeys the factor structure is studied by Fan et al. (2008).
In these cases the covariance matrix can be consistently estimated even in high-dimensional case. In the case when no additional information on the structure of the covariance matrix is available, the problem has not been studied in detail up to now. The exception is the paper of Ledoit and Wolf (2004), where a linear shrinkage estimator was suggested which possesses the smallest Frobenius loss in quadratic mean.

Ma$\breve{\text{c}}$enko and Pastur (1967), Yin (1986), Silverstein (1995), Bai et al. (2007), Bai and Silverstein (2010) used the large dimensional asymptotics to study the asymptotic behavior of the eigenvalues of general random matrices. They discovered that appropriately transformed random matrix at infinity has a nonrandom behavior and showed how to find the limiting density of its eigenvalues. In particular, Silverstein (1995) proved under very general conditions that the Stieltjes transform of the sample covariance matrix tends almost surely to a nonrandom function which satisfies some equation. This equation was first derived by Ma$\breve{\text{c}}$enko and Pastur (1967), who showed how the real covariance matrix and its sample estimate are connected at infinity. In our work we use this result for estimating functionals of the covariance matrix consistently.

In this work we concentrate on certain type of estimators, namely the shrinkage estimators. The shrinkage estimators were introduced by Stein (1956). They are constructed as a linear combination of the sample estimator and some known target. These estimators have remarkable property: they are biased but can significantly reduce the mean square error of the estimator. In the large as well as in the small dimensional cases it is difficult to find the consistent estimators for the so-called shrinkage intensities. In this situation Ledoit and Wolf (2004) made progress when the target matrix is the identity and found a feasible linear shrinkage estimator for the covariance matrix which is optimal in the sense of the squared mean. This estimator provided a remarkable dominance over the sample estimator and other known estimators for the covariance matrix. The linear shrinkage presented by Ledoit and Wolf (2004) shows its best performance in case when the eigenvalues of the covariance matrix are not dispersed and/or the concentration ratio $c$ is large.

In this paper we extend the work of Ledoit and Wolf (2004) by constructing a more general linear shrinkage estimator for a large dimensional covariance matrix.  The target matrix here is considered to be an arbitrary symmetric positive definite matrix with uniformly bounded trace norm. Using random matrix theory we prove that the optimal shrinkage intensities are nonrandom at infinity, find their asymptotic deterministic equivalents and estimate them consistently. Additionally we show that the Frobenius norm of the covariance matrix tends to a deterministic quantity which can also be estimated consistently. The resulting estimator  obeys \textit{almost surely} the smallest Frobenius loss when the dimension $p$ and the sample size $n$ increase together and $p/n\rightarrow c\in(0, \infty)$ as $n\rightarrow\infty$.

The rest of paper is organized as follows. In Section 2 we present the preliminary results from the random matrix theory which are used in the proofs of the theorems. Section 3 contains the \textit{oracle} linear shrinkage estimator and the main asymptotic results on the shrinkage intensities and the Frobenius norm of the sample covariance matrix. In Section 4 we present the \textit{bona fide} linear shrinkage estimator for the covariance matrix and make a short comparison with the well-known Ledoit and Wolf (2004) estimator. The results of the empirical study are provided in Section 5, while Section 6 summarizes all main results of the paper. The proofs of the theorems are moved to the appendix.

\section{Preliminary results and large dimensional asymptotics}

By "large dimensional asymptotics" or "Kolmogorov asymptotics" it is understood that $\dfrac{p}{n}\rightarrow c\in (0, +\infty)$ where the number of variables $p\equiv p(n)$ and the sample size $n$ both tend to infinity. In this case the traditional sample estimators perform poorly or very poorly and tend to over/underestimate the population covariance matrix.

We use the following notations in the paper:
\begin{itemize}
\item $\bSigma_n$ stands for the covariance matrix, and $\bS_n$ denotes the corresponding sample covariance matrix.
\item The pairs $(\tau_i,\bn_i)$ for $i=1,\ldots,p$ are the collection of eigenvalues and the corresponding orthonormal eigenvectors of the covariance matrix $\bSigma_n$.
\item $H_n(t)$ is the empirical distribution function (e.d.f.) of the eigenvalues of $\bSigma_n$, i.e.,
\begin{equation}\label{cov_edf}
H_n(t)=\dfrac{1}{p}\sum\limits_{i=1}^{p}\mathbbm{1}_{\{\tau_i<t\}}\,
\end{equation}
 where $\mathbbm{1}_{\{\cdot\}}$ is the indicator function.
\item Let $\bx_n$ be a $p\times n$ matrix which consists of independent and identically distributed (i.i.d.) real random variables with zero mean and unit variance such that
\begin{equation}\label{obs}
 \by_n=\bSigma_n^{\frac{1}{2}}\bx_n.
 \end{equation}
\end{itemize}
In the derivation of the main results the following five assumptions are used.
\begin{description}
\item[(A1)] The population covariance matrix $\bSigma_n$ is a nonrandom $p$-dimensional positive definite matrix.

\item[(A2)] Only the matrix $\by_n$ is observable. We know neither $\bx_{n}$ nor $\bSigma_n$ itself.

\item[(A3)] We assume that $H_n(t)$ converges to some limit $H(t)$ at all points of continuity of $H$.

\item[(A4)] The elements of the matrix $\bx_n$ have uniformly bounded moments of order $4+\varepsilon, \varepsilon>0$.

\item[(A5)] The largest eigenvalue of the covariance matrix $\bSigma_n$ is at most of the order $O(\sqrt{p})$. Moreover, we assume that the order of only finite number of eigenvalues could depend on $p$.
\end{description}

The assumptions (A1)-(A3) are important to prove Mar$\breve{\text{c}}$henko-Pastur equation (see, e.g., Silverstein (1995)) and they are standard in the large dimensional asymptotics (see, e.g., Bai and Silverstein (2010)). In particular, the assumption (A3) on the existence of the limiting population spectral function $H(t)$ is also very general because the support of $H(t)$ may be unbounded and non-compact. The assumption (A4) is needed for the proofs. This assumption is much weaker than the corresponding one used by Ledoit and Wolf (2004) for similar purposes which is the existence of the $8$th moment. The assumption (A5) ensures that the covariance matrix may have a couple of very large eigenvalues. Indeed, it is practically relevant and a good example of a model with unbounded eigenvalues of the covariance matrix is the factor model (see, Bai (2003), Bai and Ng (2002), Fan et al. (2013)). Moreover, even if the structure of the data is more complex and the factor model is not present we still fall within the proposed theoretical assumptions (A1)-(A5). To the end, we do not impose the assumption of any particular distribution on the underlying data generating process and the presented framework is also applicable if the sample covariance matrix is singular, i.e. if the sample size $n$ is larger than the dimension $p$. As a result, the assumptions (A1)-(A5) are general enough to cover many practical situations.

Let $(\lambda_i,\mathbf{u}_i)$ for $i=1,\ldots,p$ denote the set of eigenvalues and the corresponding orthonormal eigenvectors of the sample covariance matrix
 \footnote{It must be noted that the sample mean vector was omitted because the $1$-rank matrix $\sx\sx^{\prime}$ does not effect the asymptotic behavior of the spectrum of the sample covariance matrix (see, Bai and Silverstein (2010), Theorem A.44).}
\begin{equation}\label{samplecov}
 \bS_n=\dfrac{1}{n}\by_n\by_n^{\prime}=\dfrac{1}{n}\bSigma_n^{\frac{1}{2}}\bx_n\bx_n^{\prime}\bSigma_n^{\frac{1}{2}}\,.
 \end{equation}

Similar to (\ref{cov_edf}) the (e.d.f.) of the sample covariance matrix $\bS_n$ is defined by
\begin{equation}\label{sample_edf}
F_n(\lambda)=\dfrac{1}{p}\sum\limits_{i=1}^{p}\mathbbm{1}_{\{\lambda_i<\lambda\}},~~\lambda\in\mathbbm{R}\,.
\end{equation}
The most powerful tool for investigating the asymptotics of (e.d.f) $F_n(\lambda)$ is the Stieltjes transform. For a nondecreasing function $G$  with bounded variation the Stieltjes transform is defined by
\begin{equation}\label{Strans}
\forall z\in\mathbbm{C}^+~~~m_G(z)=\int\limits_{-\infty}^{+\infty}\dfrac{1}{\lambda-z}dG(\lambda)\,.
\end{equation}
In our notations  $\mathbbm{C}^+=\{z\in\mathbbm{C}: \textbf{Im}(z)>0\}$ is a half-plane of complex numbers with strictly positive imaginary part and any complex number is given by $z=\textbf{Re}(z)+i\textbf{Im}(z)$ with $\textbf{Re}(z)$ and $\textbf{Im}(z)$ the real and the imaginary parts accordingly. The Stieltjes transform and its importance for the behavior of spectrum of large dimensional random matrices is discussed in Silverstein (2009) in detail.

 For all $z\in\mathbbm{C}^+$ the Stieltjes transform of the sample (e.d.f.) $F_n(\lambda)$ is given by
\begin{eqnarray}\label{trans_s}
m_{F_n}(z)&=&\dfrac{1}{p}\sum\limits_{i=1}^{p}\int\limits_{-\infty}^{+\infty}\dfrac{1}{\lambda-z}\delta(\lambda-\lambda_i)d\lambda=\dfrac{1}{p}\Tr\{(\bS_n-z\bI)^{-1}\}\,,
\end{eqnarray}
where $\bI$ is a suitable identity matrix, tr($\cdot$) is the trace of the matrix, and $\delta(\cdot)$ is the Dirac delta function.

Mar$\breve{\text{c}}$enko and Pastur (1967) proved that the (e.d.f.) $F_n(\lambda)$ converges almost surely (a.s.) to a nonrandom limit $F(\lambda)$. Moreover, they derived an equation, the so-called Mar$\breve{\text{c}}$enko-Pastur (MP) equation, which shows the connection between $F(\lambda)$ and $H(\tau)$ at infinity. Under more general conditions the MP equation was considered by several authors (see, Yin (1986), Silverstein and Choi (1995) and Silverstein and Bai (1995)). The most general case was presented by Silverstein (1995) where the strong convergence under very general conditions was established. For illustration purposes, we summarize this result in Theorem 2.1.
\begin{theorem}{\textbf{[Silverstein (1995)]}} Assume that assumptions (A1)-(A3) are satisfied on the common probability space and $\frac{p}{n}\rightarrow c\in(0,+\infty)$ as $n\rightarrow\infty$. Then $F_n(t)\stackrel{a.s.}{\Rightarrow}F(t)$ as $n\rightarrow\infty$. Moreover, the Stieltjes transform of $F$ satisfies the following equation
\begin{equation}\label{MP}
m_F(z)=\int\limits_{-\infty}^{+\infty}\dfrac{1}{\tau(1-c-czm_F(z))-z}dH(\tau)\,,
\end{equation}
in the sense that $m_F(z)$ is the unique solution of (\ref{MP}) for all $z\in\mathbbm{C}^+$.
\end{theorem}
The MP equation (\ref{MP}) has a closed-form solution in several restricted cases. The famous Mar$\breve{\text{c}}$enko-Pastur law appears only when the covariance matrix $\bSigma_n$ is the multiple of the identity matrix, i.e., $\bSigma_n=\sigma\bI$.

\section{Optimal shrinkage estimator for covariance matrix}

In this section we construct an optimal shrinkage estimator for the covariance matrix under high-dimensional asymptotics. This estimator is only \textit{oracle} one, i.e., it depends on unknown quantities. The corresponding \textit{bona fide} estimator is given in Section 4.

 Note that the linear shrinkage estimator of Ledoit and Wolf (2004) has the smallest Frobenius loss only in quadratic mean. Consequently, it would be a great advantage to construct an estimator with almost sure smallest Frobenius loss. The general linear shrinkage estimator (GLSE) for the covariance matrix is given by
 \begin{equation}\label{gse1}
\widehat{\bSigma}_{GLSE}=\alpha_n\bS_n+\beta_n\bSigma_0~~\text{with}~||\bSigma_0||_{tr}\leq M\,.
\end{equation}
 where the symmetric positive definite matrix $\bSigma_0$ has bounded trace norm at infinity, i.e., there exists $M>0$ such that $\sup\limits_{n}||\bSigma_0||_{tr}=\sup\limits_{n}\text{tr}(\bSigma_0)\leq M$. No assumption is imposed on the shrinkage intensities $\alpha_n$ and $\beta_n$ which are objects of our interest.

  Since the shrinkage intensities are the main object of investigation and $\bSigma_0$ is fixed, a question arises how to choose this target matrix. It could depend on the underlying data as well as on the branch of science where the shrinkage estimator for the covariance matrix is applied, i.e., wireless communications, finance etc. On the other hand, from the view of Bayesian statistics $\bSigma_0$  can be interpreted as a hyperparameter of \textit{a priori} distribution. The useful comments about the choice of the target matrix $\bSigma_0$ can be found e.g. in Bai and Shi (2011), Ledoit and Wolf (2004).

The aim is to find the optimal shrinkage intensities $\alpha_n$ and $\beta_n$ which minimize the Frobenius norm for a given nonrandom target matrix $\bSigma_0$ expressed as
\begin{equation}\label{risk1}
L^2_F=||\widehat{\bSigma}_{GLSE}-\bSigma_n||_F^2=||\bSigma_n||^2_F+||\widehat{\bSigma}_{GLSE}||^2_F-2\text{tr}(\widehat{\bSigma}_{GLSE}\bSigma_n)\,.
\end{equation}
 As a result, using (\ref{gse1}) we want to solve the following optimization problem
\begin{eqnarray}\label{optm1}
&&L^2_F=\alpha_n^2||\bS_n||^2_F+2\alpha_n\beta_n\text{tr}(\bS_n\bSigma_0)+\beta_n^2||\bSigma_0||^2_F\\
&&-2\alpha_n\text{tr}(\bS_n\bSigma_n)-2\beta_n\text{tr}(\bSigma_n\bSigma_0)\longrightarrow\text{min}\nonumber\\
&&~~\text{with respect to}~\alpha_n~\text{and}~\beta_n\nonumber\,.
\end{eqnarray}
Taking the derivatives of $L^2_F$ with respect to $\alpha_n$ and $\beta_n$ and setting them equal to zero we get
\begin{equation}\label{dera}
\dfrac{\partial L^2_F}{\partial\alpha_n}=\alpha_n||\bS_n||^2_F+\beta_n\text{tr}(\bS_n\bSigma_0)-\text{tr}(\bS_n\bSigma_n)=0\,,
\end{equation}
\begin{equation}\label{derb}
\dfrac{\partial L^2_F}{\partial\beta_n}=\alpha_n\text{tr}(\bS_n\bSigma_0)+\beta_n||\bSigma_0||^2_F-\text{tr}(\bSigma_n\bSigma_0)=0\,.
\end{equation}
The Hessian of the $L^2_F$ has the form
\begin{equation}\label{hessian1}
 \bH=\left(
 \begin{array}{cc}
 ||\bS_n||^2_F&\text{tr}(\bS_n\bSigma_0)\\
 \text{tr}(\bS_n\bSigma_0)&||\bSigma_0||^2_F\,
 \end{array}
 \right)\,.
\end{equation}
From the following inequality it follows that the Hessian matrix $\bH$ is always positive definite:
\begin{align}\label{posdef} \text{det}(\bH)=||\bS_n||^2_F||\bSigma_0||^2_F-(\text{tr}(\bS_n\bSigma_0))^2&\geq||\bS_n||^2_F||\bSigma_0||^2_F-||\bS_n||^2_{\infty}(\text{tr}(\bSigma_0))^2\nonumber\\
 &\hspace{-3mm}\overset{Jensen}{\geq}(||\bS_n||^2_F-||\bS_n||^2_{\infty})||\bSigma_0||^2_F>0\,,
\end{align}
where $||\bS_n||_{\infty}$ denotes the spectral norm (square root of maximum eigenvalue of matrix $\bS_n^2$). The last inequality in (\ref{posdef}) is well-known (see, e.g., Golub and Van Loan (1996)).

From equations (\ref{dera}) and (\ref{derb}) it is easy to find the optimal shrinkage intensities $\alpha_n^*$ and $\beta_n^*$ as
 \begin{equation}\label{alfa} \alpha_n^*=\dfrac{\text{tr}(\bS_n\bSigma_n)||\bSigma_0||^2_F-\text{tr}(\bSigma_n\bSigma_0)\text{tr}(\bS_n\bSigma_0)}{||\bS_n||^2_F||\bSigma_0||^2_F-\bigl(\text{tr}(\bS_n\bSigma_0)\bigr)^2}\,,
 \end{equation}
 \begin{equation}\label{beta} \beta_n^*=\dfrac{\text{tr}(\bSigma_n\bSigma_0)||\bS_n||^2_F-\text{tr}(\bS_n\bSigma_n)\text{tr}(\bS_n\bSigma_0)}{||\bS_n||^2_F||\bSigma_0||^2_F-\bigl(\text{tr}(\bS_n\bSigma_0)\bigr)^2}\,.
 \end{equation}
  Next, we consider the asymptotic behavior of the quantities (\ref{alfa}) and (\ref{beta}), namely we look for their asymptotic equivalents. Recall that the sequence of random variables $\tilde{\xi}_n$ is called asymptotically equivalent to a nonrandom sequence $\xi_n$ when
  \begin{equation}
   \left|\tilde{\xi}_n-\xi_n\right|\longrightarrow0~~\text{a. s.}~\text{for}~n\rightarrow\infty\,.
  \end{equation}
  Note that it is sufficient to know the asymptotic equivalents of the quantities $||\bS_n||^2_F$, $\text{tr}(\bS_n\bSigma_n)$ and $\text{tr}(\bS_n\bSigma_0)$. It is not difficult to find the asymptotic equivalents to the last two quantities. This is done in Theorem 3.2. More difficult is to find a asymptotic equivalent to the first quantity, namely $||\bS_n||^2_F$. Since the Frobenius norm of the sample covariance matrix is very important in high dimensional statistics, it would be a great advantage to investigate its asymptotic behavior.

  In Theorem 3.1 we present our first result where we show that the normalized Frobenius norm of $\bS_n$ tends almost surely to a nonrandom quantity for $p/n\rightarrow c\in(0,+\infty)$ as $n\rightarrow\infty$.

\begin{theorem} Assume that (A1)-(A5) hold and $\dfrac{p}{n}\rightarrow c\in(0, +\infty)$ for $n\rightarrow\infty$. Then the Frobenius norm of the sample covariance matrix $\phi_n=\dfrac{1}{p}||\bS_n||^2_F$ almost surely tends to a nonrandom variable $\phi$ which is given by
\begin{equation}\label{phi}
 \phi=\int\limits_{-\infty}^{+\infty}\tau^2dH(\tau)+c\left(\int\limits_{-\infty}^{+\infty}\tau dH(\tau)\right)^2\,,
\end{equation}
where $H(t)$ denotes the limiting function of the spectral e.d.f. $H_n(t)$ defined in (\ref{cov_edf}).
\end{theorem}

The proof of Theorem 3.1 follows from the MP equation (\ref{MP}) and is presented in the Appendix.  In particular, it ensures that the Frobenius norm of the sample covariance matrix is fixed and depends solely on the function $H$ and $c$ on infinity. Nevertheless, we are not able to estimate $||\bSigma_n||^2_F$ using this knowledge since the function $H(t)$ is unknown. That is why we need to find the asymptotically equivalent quantity to $||\bS_n||^2_F$.

Theorem 3.1 gives us some intuition about this quantity. The idea is to replace the integrals in (\ref{phi}) by the corresponding finite sums, namely $1/p\sum\limits_{i=1}^p\tau_i^2=1/p\text{tr}(\bSigma_n^2)$ and $1/p\sum\limits_{i=1}^p\tau_i=1/p\text{tr}(\bSigma_n)$. The main advantage of the procedure is that this substitution does not effect the almost sure convergence.

\begin{theorem} Under assumptions (A1)-(A5) for $\dfrac{p}{n}\rightarrow c\in(0, +\infty)$ it holds that
\begin{equation}\label{phi2}
\dfrac{1}{p}\Biggl|||\bS_n||^2_F-\biggl(||\bSigma_n||^2_F+\dfrac{c}{p}||\bSigma_n||^2_{tr}\biggr)\Biggr|\longrightarrow0~~\text{a.s.}~\text{for}~n\rightarrow\infty\,
\end{equation}
where $||\bSigma||^2_{tr}=\bigl(\text{tr}(\bSigma_n)\bigr)^2$ is the squared trace norm of matrix $\bSigma_n$.

Additionally, for the quantity $\text{tr}(\bS_n\mathbf{\Theta})$, where $\mathbf{\Theta}$ is a symmetric positive definite matrix with bounded trace norm, it holds  that
\begin{equation}\label{trst}
\dfrac{1}{p}\Biggl|\text{tr}(\bS_n\mathbf{\Theta})-\text{tr}(\bSigma_n\mathbf{\Theta})\Biggr|\longrightarrow0~~\text{a. s.}~~\text{for}~~\dfrac{p}{n}\rightarrow c\in(0, +\infty)~\text{as}~n\rightarrow\infty\,.
\end{equation}
\end{theorem}

The proof of the theorem is given in the Appendix. In contrast to Theorem 3.1, Theorem 3.2 contains a better interpretation of the asymptotic result (\ref{phi2}).
It shows, in particular, that the consistent estimator of the Frobenius norm of the real covariance matrix, $||\bSigma_n||^2_F$, is not equal to its sample counterpart. On the other hand the functionals of the type $\text{tr}(\bS_n\mathbf{\Theta})$ are consistently estimated by the sample counterparts.

Moreover, it appears that under the large dimensional asymptotics, the Frobenius norm of $\bS_n$ is shifted by the constant $c/p||\bSigma||^2_{tr}$. As a result, we know the asymptotic value of the bias and the consistent estimator for Frobenius norm of the covariance matrix can be constructed by subtracting it from the sample counterpart. The exact form of this estimator is presented in Section 4.

Next, we consider the shrinkage intensities $\alpha_n^*$ and $\beta_n^*$. The application of Theorem 3.2 allows us to find their asymptotic properties.
This is done in Corollary 3.1.
\begin{corollary} Assume that (A1) - (A5) are fulfilled. Then for $\dfrac{p}{n}\rightarrow c\in(0, +\infty)$ as $n\rightarrow\infty$ the optimal shrinkage intensities $\alpha_n^*$ and $\beta_n^*$ satisfy
\begin{equation}\label{alfa2}
\Biggl|\alpha_n^*-\alpha^*\Biggr|\longrightarrow0~~\text{a.s.}~\text{for}~n\rightarrow\infty\,,
\end{equation}
 where
 \begin{equation} \alpha^*=1-\dfrac{\dfrac{c}{p}||\bSigma||^2_{tr}||\bSigma_0||^2_F}{\bigl(||\bSigma_n||^2_F+\dfrac{c}{p}||\bSigma_n||^2_{tr}\bigr)||\bSigma_0||^2_F-\bigl(\text{tr}(\bSigma_n\bSigma_0)\bigr)^2}
 \end{equation}
and
\begin{equation}\label{beta2}
\Biggl|\beta_n^*-\beta^*\Biggr|\longrightarrow0~~\text{a.s.}~\text{for}~n\rightarrow\infty\,,
\end{equation}
 where
 \begin{equation}
  \beta^*=\dfrac{\text{tr}(\bSigma_n\bSigma_0)}{||\bSigma_0||^2_F}\left(1-\alpha^*\right)\,.
 \end{equation}
\end{corollary}
\begin{proof}
This result follows straightforwardly from Theorem 3.2. Only the boundedness of the trace norm of the matrix $1/p\bSigma_n$ is needed. It obviously follows from the boundedness of its spectral norm $||\bSigma_n||_{\infty}$, namely
\begin{equation}\label{cor1}
||1/p\bSigma_n||_{tr}=1/p\text{tr}(\bSigma_n)\stackrel{Jensen}{\leq}\dfrac{1}{\sqrt{p}}||\bSigma_n||_{F}\leq\sqrt{||\bSigma_n||_{\infty}}<\infty\,.
\end{equation}
\end{proof}
First, we observe that the application of Corollary 3.1 and Theorem 3.2 ensures that the shrinkage intensities $\alpha^*$ and $\beta^*$ can be consistently estimated. We summarize this result in Section 4.

The second interpretation of the results of Corollary 3.1 shows that $\alpha^*<1$ as soon as $c>0$ which follows directly  from (\ref{posdef}). Only if $c=0$, we get $\alpha^*=1$ and $\beta^*=0$. In this case, the sample covariance matrix is a good estimator for the population covariance matrix which minimizes the Frobenius norm. In contrast if $p$ increases, i.e. if $c>0$, the general linear shrinkage estimator (\ref{gse1}) improves the performance of the sample estimator. Moreover, the impact of this improvement becomes larger as $p$ approaches $n$.

\section{Bona fide estimator for the covariance matrix}
This section is dedicated to the \textit{bona fide} estimation of the unknown parameters from Section 3. As we have already mentioned the asymptotic shrinkage intensities $\alpha^*$ and $\beta^*$ from Corollary 3.1 can be consistently estimated using the result of Theorem 3.2.

Obviously, both $\alpha^*$ and $\beta^*$ depend on the Frobenius norm of the covariance matrix $\bSigma_n$ and on the functionals of the type $\text{tr}(\bSigma_n\bTheta)$. Due to Theorem 3.2 the consistent estimator of latter term is its sample counterpart, while the consistent estimator of  the Frobenius norm $\psi_n=\dfrac{1}{p}||\bSigma_n||^2_F$ is given by
\begin{equation}\label{Fnormest}
\widehat{\psi}_n=\dfrac{1}{p}||\bS_n||^2_F-\dfrac{1}{np}||\bS_n||^2_{tr}\,.
\end{equation}
 The resulting estimator (\ref{Fnormest}) is similar to that obtained by Girko (1995). In contrast to Girko (1995), Theorem 3.2 was proved under more general conditions and we use the complex Stieltjes transform while Girko (1995) used the real one. Girko's result holds for $c\in(0,1)$ while our for $c\in(0, +\infty)$. Nevertheless, using Theorem 3.1 and Theorem 3.2 it can be shown that the so-called $G^2_4$-estimator considered by Girko (1995) coincides with (\ref{Fnormest}) and it is indeed the consistent estimator for the Frobenius norm of the covariance matrix under the large dimensional asymptotics.

The optimal linear shrinkage estimator (OLSE) for the covariance matrix $\bSigma_n$ is given by
 \begin{equation}\label{optshrin}
\widehat{\bSigma}_{OLSE}=\hat{\alpha}^*\bS_n+\hat{\beta}^*\bSigma_0~~\text{with}~||\bSigma_0||_{tr}\leq M\,,
\end{equation}
where
\begin{equation}\label{hata}
 \hat{\alpha}^*=1-\dfrac{\dfrac{1}{n}||\bS_n||^2_{tr}||\bSigma_0||^2_F}{||\bS_n||^2_F||\bSigma_0||^2_F-\bigl(\text{tr}(\bS_n\bSigma_0)\bigr)^2}
 \end{equation}
and
 \begin{equation}\label{bata}
  \hat{\beta}^*=\dfrac{\text{tr}(\bS_n\bSigma_0)}{||\bSigma_0||^2_F}\left(1-\hat{\alpha}^*\right)\,.
 \end{equation}

 The OLSE estimator possesses almost surely the smallest Frobenius loss according to Theorem 3.2 and has a simple structure. Moreover, when $p>n$ and the sample covariance matrix $\bS_n$ is singular the optimal linear shrinkage $\widehat{\bSigma}_{OLSE}$ stays invertible and applicable in practice.

Next, we consider an interesting special case when $\bSigma_0=\frac{1}{p}\bI$ in more details. In this case $\widehat{\bSigma}_{OLSE}$ looks very similar to the linear shrinkage estimator proposed by Lediot and Wolf (2004). However, they are not equal.

First, the linear shrinkage estimator of Lediot and Wolf (2004) has the smallest Frobenius loss in quadratic mean while the suggested estimator in (\ref{optshrin}) ensures almost sure convergence to the oracle. Moreover, Lediot and Wolf (2004) assumed the existence of $8$th moment while our estimator is derived under the assumption that $4+\varepsilon, \varepsilon>0$, moment exists.

Second, the estimator of Ledoit and Wolf (2004) and the suggested optimal linear shrinkage estimator (\ref{optshrin}) with $\bSigma_0=1/p\bI$ differ in $\hat{\alpha}^*$. Instead of $\dfrac{1}{n}||\bS_n||^2_{tr}$ Ledoit and Wolf (2004) used $\dfrac{1}{n^2}\sum\limits_{i=1}^n||\mathbf{y}_i\mathbf{y}^\prime_i-\bS_n||^2_F$, where $\mathbf{y}_i$ are the columns of the matrix $\bSigma_n^{1/2}\mathbf{X}_n$. Indeed, let $d^2=\dfrac{1}{p}||\bS_n||^2_F-\bigl(1/p\text{tr}(\bS_n)\bigr)^2$ and $b^2=\dfrac{1}{p}\dfrac{1}{n^2}\sum\limits_{i=1}^n||\mathbf{y}_i\mathbf{y}^\prime_i-\bS_n||^2_F$. Then the estimate for $\alpha^*$ by Ledoit and Wolf (2004) is given by
\begin{equation}\label{LW2004}
\hat{\alpha}^*_{LW}=1-\dfrac{min\{b^2,d^2\}}{d^2}\,.
\end{equation}
Consequently, from (\ref{LW2004}) we observe that the Ledoit-Wolf (LW) linear shrinkage estimator is constrained whereas our optimal shrinkage estimator (\ref{optshrin}) for $\bSigma_0=\frac{1}{p}\bI$ is unconstrained. Moreover, if $b^2>d^2$ in (\ref{LW2004}) then $\alpha^*_{LW}=0$ independently how large $p$ is with respect to $n$. In this case LW estimator is equal to the target matrix $\text{tr}(\bS_n)\dfrac{1}{p}\bI$.  In contrast, for the suggested OLSE estimator (\ref{optshrin}) it holds always that $0<\alpha^*\leq1$. Moreover, $\alpha^*=1$ iff $c=0$ which means that the sample covariance matrix possesses the smallest Frobenius loss only if $p$ is much smaller than $n$. For $c>0$, the sample covariance matrix is not an optimal estimator for the covariance matrix.

 Third, the LW linear shrinkage estimator seems to be more computationally intensive than (\ref{optshrin}) for $\bSigma_0=\frac{1}{p}\bI$. The reason is that the quantity $b^2$ is computed by a loop while $\dfrac{1}{n}||\bS_n||^2_{tr}$ needs only the computation of the trace.

\begin{figure}[H]
\includegraphics[scale=0.36]{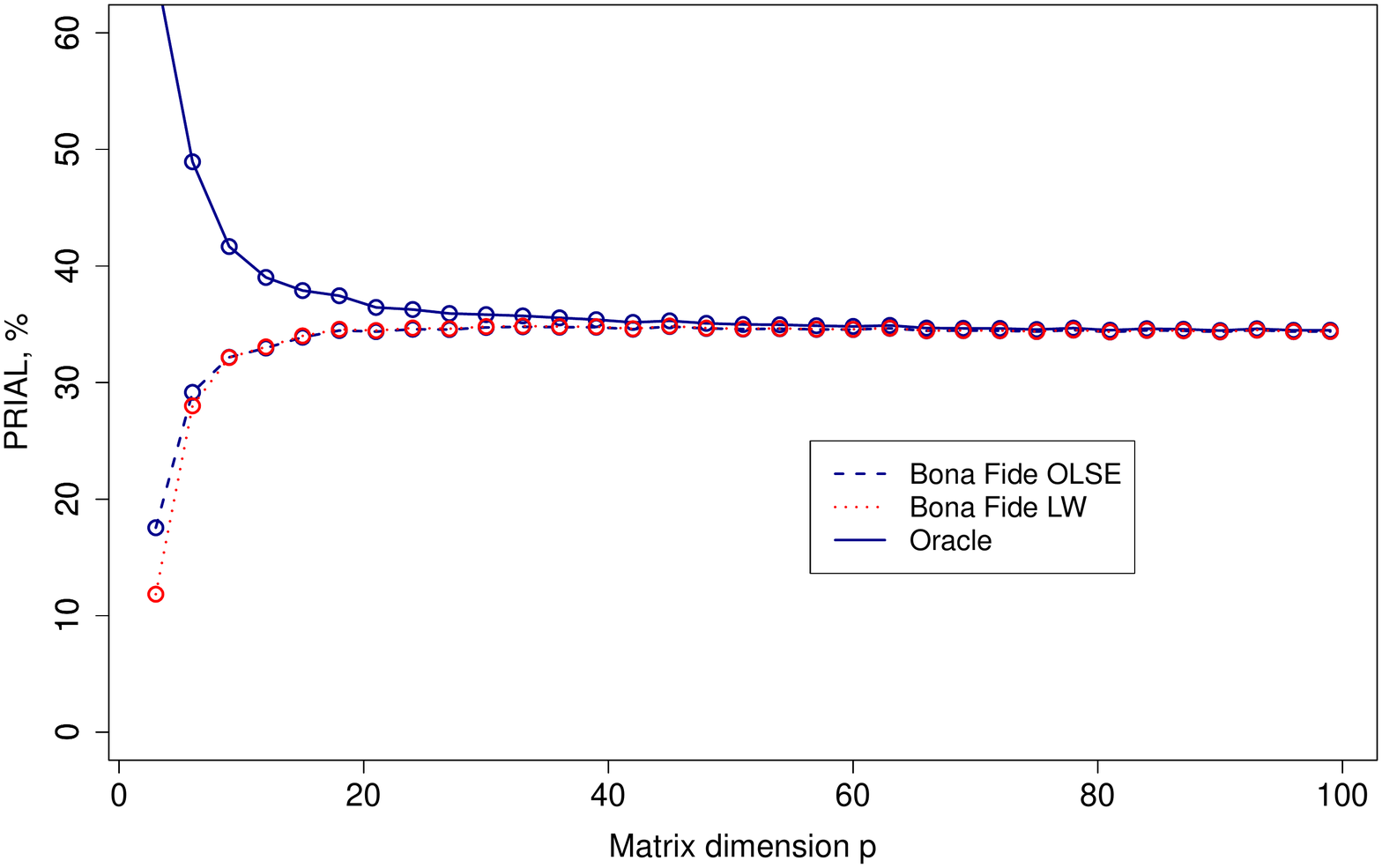}
\caption{PRIALs for the oracle, the bona fide OLSE, and the bona fide LW estimators for $p=3k,k\in\{1,\ldots,33\}$, $c=1/3$, and $\bSigma_0=\frac{1}{p}\bI$, 1000 repetitions.}
\label{Fig:1}
\end{figure}

Ledoit and Wolf (2004) proved that their linear shrinkage estimator tends on average to the oracle one while our estimator tends to the oracle almost surely and, consequently, it is expected that they are asymptotically the same. Moreover, the choice $\bSigma_0=\dfrac{1}{p}\bI$ can be too conservative and better results can be obtained for other values of $\bSigma_0$.

 In Figure 1 we present the simulation results for normally distributed data in the case when $\bSigma_0=\dfrac{1}{p}\bI$. We take, without loss of generality, $\bSigma_n$ as a diagonal matrix and separate its spectrum in three equal parts with eigenvalues $0.1$, $5$ and $10$. In terms of the corresponding cumulative distribution function of the eigenvalues of $\bSigma_n$ (ct. Section 2) it holds that
 \begin{equation}\label{sigmaH}
 H^{\bSigma_n}_n(t)=1/3\delta_{[0.1,\hspace{1mm}\infty)}(t)+1/3\delta_{[5,\hspace{1mm}\infty)}(t)+1/3\delta_{[10,\hspace{1mm}\infty)}(t)\,,
 \end{equation}
 where $\delta$ is the Dirac delta function. Doing so we leave the structure of population covariance matrix unchanged for all dimensions $p$. Then, we compare the LW linear shrinkage estimator with the suggested OLSE estimator in terms of their PRIAL's. For an arbitrary estimator of the covariance matrix, $\widehat{\mathbf{M}}$, the PRIAL (Percentage Relative Improvement in Average Loss) is defined as
 \begin{equation}\label{PRIAL}
 \text{PRIAL}(\widehat{\mathbf{M}})=\left(1-\dfrac{E||\widehat{\mathbf{M}}-\bSigma_n||^2_F}{E||\bS_n-\bSigma_n||_F^2}\right)\cdot100\%\,.
 \end{equation}
 By definition (\ref{PRIAL}), PRIAL($\bS_n$) is equal to zero and PRIAL($\bSigma_n$) is equal to 100$\%$.

Figure 1 clearly shows that both the LW estimator and the suggested OLSE estimator converge quickly to their common oracle in average. Moreover we conclude that there is no significant difference between the two estimators.

 In Figure 2 we show how the prior knowledge of the structure of the population covariance matrix can improve the OLSE estimator when $\bSigma_0=1/p\bI$. We assume that the prior matrix $\bSigma_0$ conforms to the spectrum separation of the covariance matrix $\bSigma_n$. Thus, suppose we know that the spectrum of population covariance matrix is separated in three equal blocks (see equality (\ref{sigmaH})) and we do not take any other information into account. The diagonal elements of prior matrix $\bSigma_0$ are chosen to be $1$, $2$ and $3$. In terms of the cumulative distribution function of $\bSigma_0$ it holds that
 \begin{equation}\label{sigmaH0}
 H^{\bSigma_0}(t)=1/3\delta_{[1,\hspace{1mm}\infty)}(t)+1/3\delta_{[2,\hspace{1mm}\infty)}(t)+1/3\delta_{[3,\hspace{1mm}\infty)}(t)\,.
 \end{equation}

\begin{figure}[H]
\includegraphics[scale=0.36]{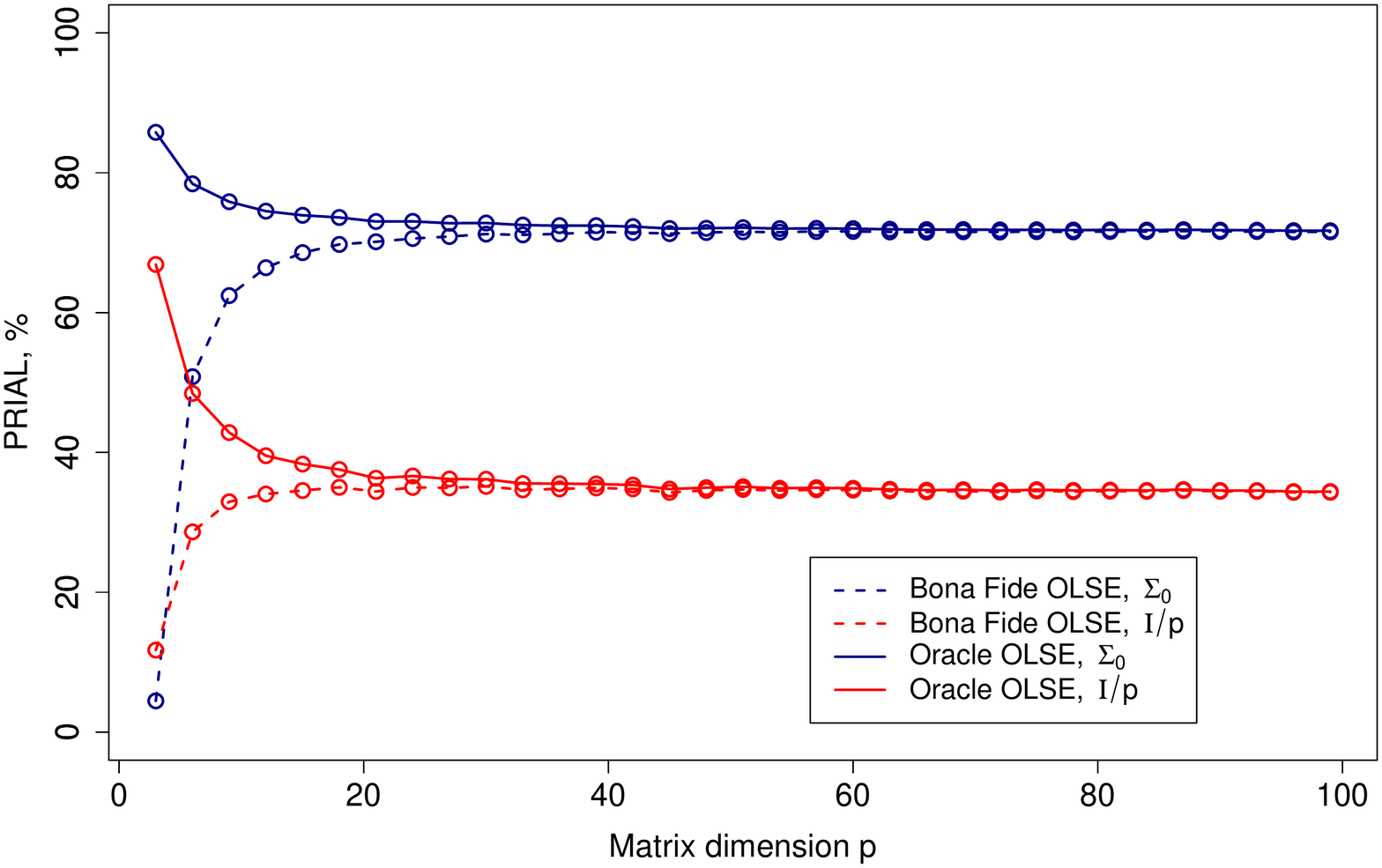}
\caption{PRIALs for the oracle and the bona fide OLSE estimators with $\bSigma_0=1/p\bI$ and $\bSigma_0$ as given in (\ref{sigmaH0}) for $p=3k,k\in\{1,\ldots,33\}$ and $c=1/3$, 1000 repetitions.}
\label{Fig:2}
\end{figure}
	
\begin{figure}[H]
\includegraphics[scale=0.36]{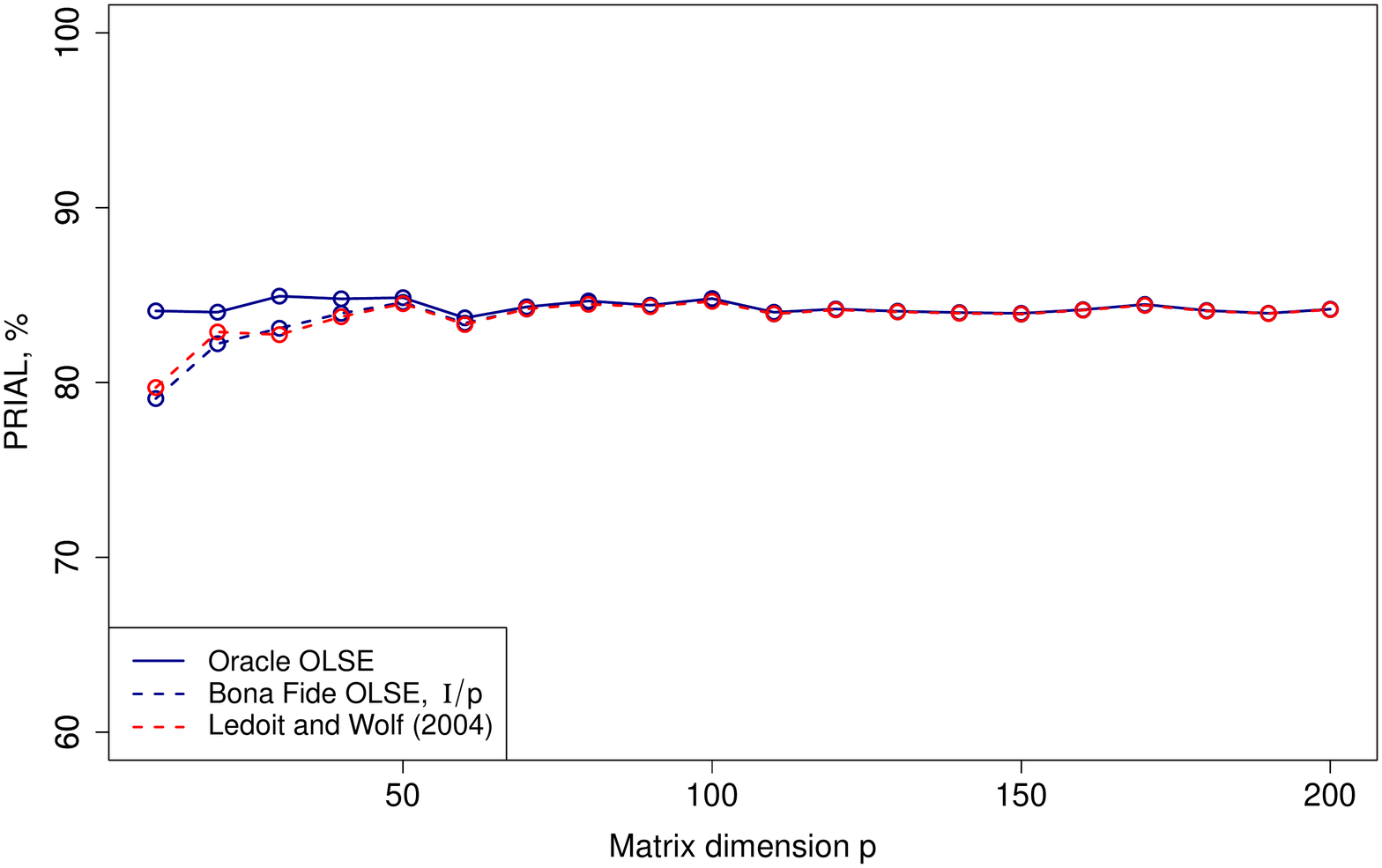}
\caption{PRIALs for the oracle, the bona fide OLSE, and the bona fide LW estimators for $p=10k,k\in\{1,\ldots,20\}$, $c=2$, and $\bSigma_0=\frac{1}{p}\bI$, 1000 repetitions.}
\label{Fig:3}
\end{figure}

In this case Figure 2 shows that improvement of the OLSE estimator is about $40\%$ over the OLSE estimator constructed with $\bSigma_0=1/p\bI$. It is remarkable that both the oracle and bona fide OLSE estimators dominate the corresponding oracle and bona fide estimators calculated for $\bSigma_0=1/p\bI$. Note that both the estimators now possess different oracles due to the prior $\bSigma_0$.

\begin{figure}[H]
\includegraphics[scale=0.36]{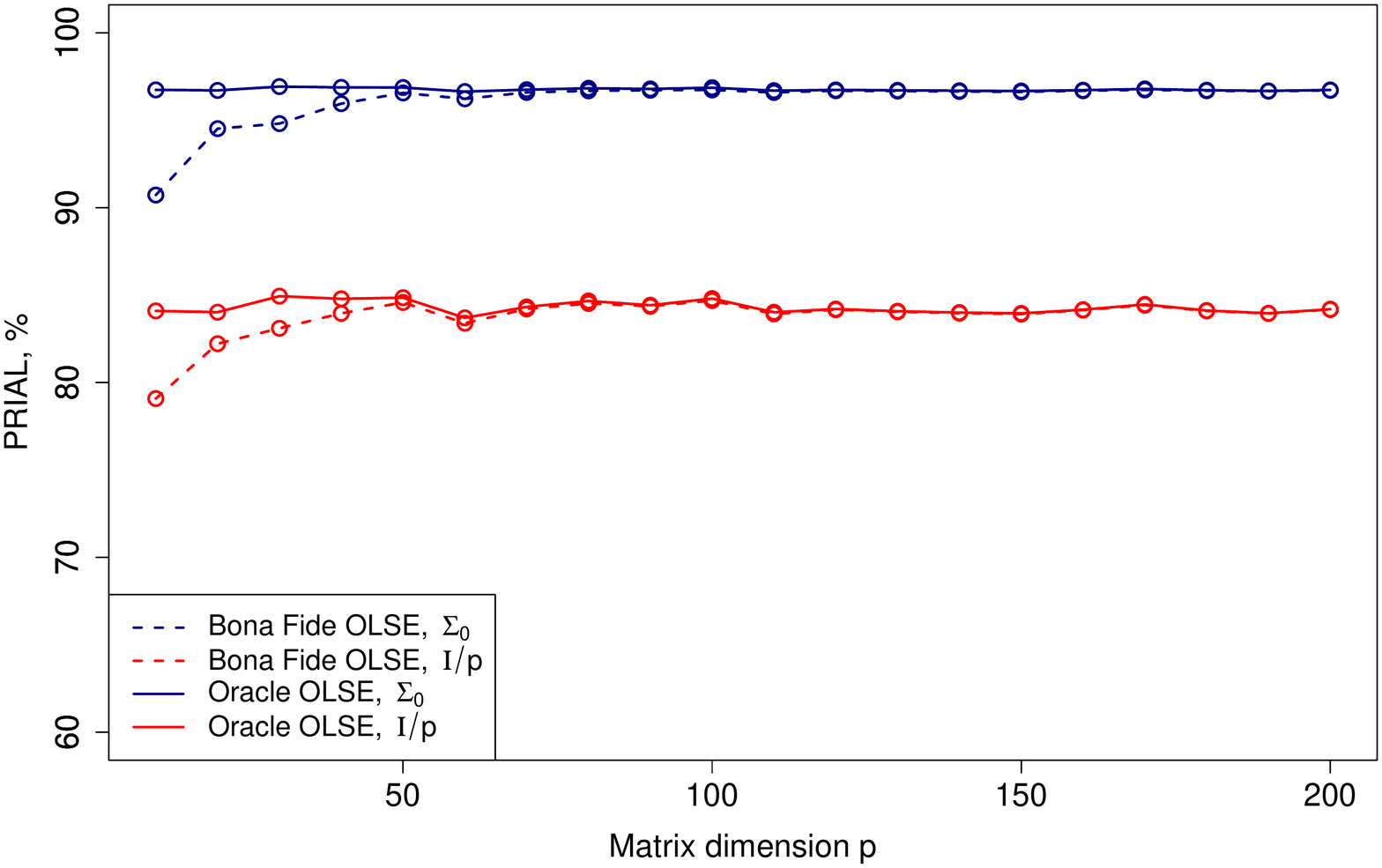}
\caption{PRIALs for the oracle and the bona fide OLSE estimators with $\bSigma_0=1/p\bI$ and $\bSigma_0$ as given in (\ref{sigmaH0}) for $p=10k,k\in\{1,\ldots,20\}$ and $c=2$, 1000 repetitions.}
\label{Fig:4}
\end{figure}

Similar situation can be observed in the case $c=2$. Here the proportion of zero eigenvalues is asymptotically equal to $1-c^{-1}=0.5$, i.e. $50\%$ (see, e.g. Bai and Silverstein (2010)). In Figures 3 and 4 we present the corresponding results of the Monte Carlo simulations. Figure 3 shows that the derived estimator for the population covariance matrix with the prior $\bSigma_0=1/p\bI$ is on average asymptotically the same as the linear shrinkage of Ledoit and Wolf (2004). It is noted that the overall performance of the OLSE estimator is significantly better than in the case $c=1/3$.

Figure 4 presents the case when the additional information about the spectrum separation from (\ref{sigmaH0}) is taken into account. Here, the spectrum of $\bSigma_0$ is determined by \[H^{\bSigma_0}(t)=1/3\delta_{[1,\hspace{1mm}\infty)}(t)+1/3\delta_{[2,\hspace{1mm}\infty)}(t)+1/3\delta_{[60,\hspace{1mm}\infty)}(t).\]
The dominance of about $20\%$ is detected. It means that the information about the spectrum separation of the population covariance matrix plays a crusial role for choosing the prior target matrix for the OLSE estimator. Consequently, when the target matrix $\bSigma_0$ departs away from the identity matrix, it has a positive influence on the OLSE estimator and can improve it significantly.

 Finally, we note that the optimal linear shrinkage estimator (\ref{optshrin}) seems to be a good generalization and modification of the linear shrinkage estimator presented by Ledoit and Wolf (2004), which is also a good alternative to the sample covariance matrix. This continues to be true even when the dimension $p$ is greater than the sample size $n$.

\section{Empirical Study}

In this section, we apply the derived estimator for the covariance matrix to real data which consist of daily asset returns on the $431$ assets listed in the S\&P 500 (Standard \& Poor's $500$) index and traded during the whole period from 13.01.2004 to 10.01.2014. The S\&P 500 is based on the market capitalizations of 500 large companies having common stock listed on the NYSE or NASDAQ. The number of the considered assets reflects the common setting in high-dimensional portfolio problems.

Next we analyse the influence of the portfolio size and the sample size on the behavior of the estimators for the Frobenius norm and the maximum and minimum eigenvalues of the population covariance matrix which are based on the derived OLSE and the sample covariance matrix. In Figure \ref{Fig:5} we present the results in case of $p=156$ and $n \in \{104, 130, 195, 312\}$ which leads to $c \in \{0.5, 0.8, 1.2, 1.5\}$, respectively. Since the choice of the assets is not unique, here we sample randomly $p=156$ assets out of $431$ and generate $10^3$ different portfolios. In Figures \ref{Fig:6} and \ref{Fig:7}, the results are shown for several values of $p \in \{50, 100, 200, 300\}$. The sample size $n$ is chosen such that $c \in (0,3)$. For each point from the figure the maximum (minimum) eigenvalue is calculated based on a randomly chosen portfolio of the dimension $p$.

\begin{figure}[H]
\includegraphics[scale=0.5]{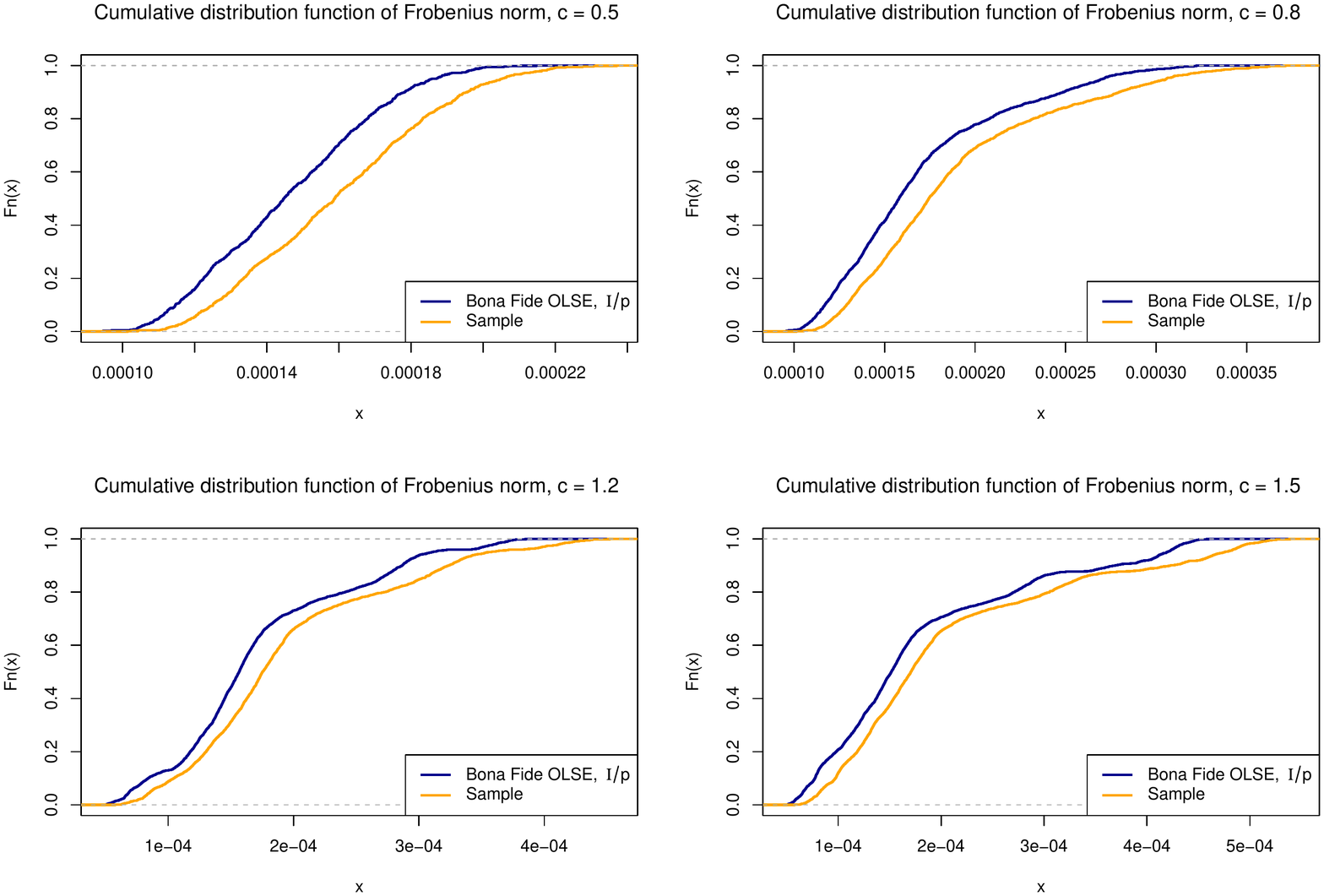}
\caption{Empirical distribution function of the Frobenius norms of the OLSE and the sample covariance matrix.}
\label{Fig:5}
\end{figure}

Figure \ref{Fig:5} provides the empirical distribution of the Frobenius norms of the OLSE and the sample covariance matrix in case of $10^3$ randomly chosen portfolios of the dimension $p=156$. Here, we observe that the Frobenius norm of the OLSE is uniformly smaller than the one of the sample covariance matrix. This result holds true for all of the considered values of $c \in \{0.5, 0.8, 1.2, 1.5\}$. It is in line with the results of Theorem 3.2 where it is proved that the Frobenius norm of the sample covariance matrix overestimates asymptotically the corresponding population value.

In Figures \ref{Fig:6} and \ref{Fig:7}, we analyze the behavior of the smallest and the largest eigenvalues of the OLSE and the sample estimator of the covariance matrix. The results in case of the smallest eigenvalue are of great importance in Finance since the smallest eigenvalue of the covariance matrix is directly related to the variance of the global minimum variance portfolio which is very popular in Portfolio Theory (cf. Frahm and Memmel (2010)). We observe that the OLSE makes the largest eigenvalue smaller, whereas the smallest eigenvalue becomes larger in comparison to the sample covariance matrix. This result allows us to correct the investor's overoptimism concerning the risk when the global minimum variance is constructed by using the sample covariance matrix.

\begin{figure}[H]
\includegraphics[scale=0.5]{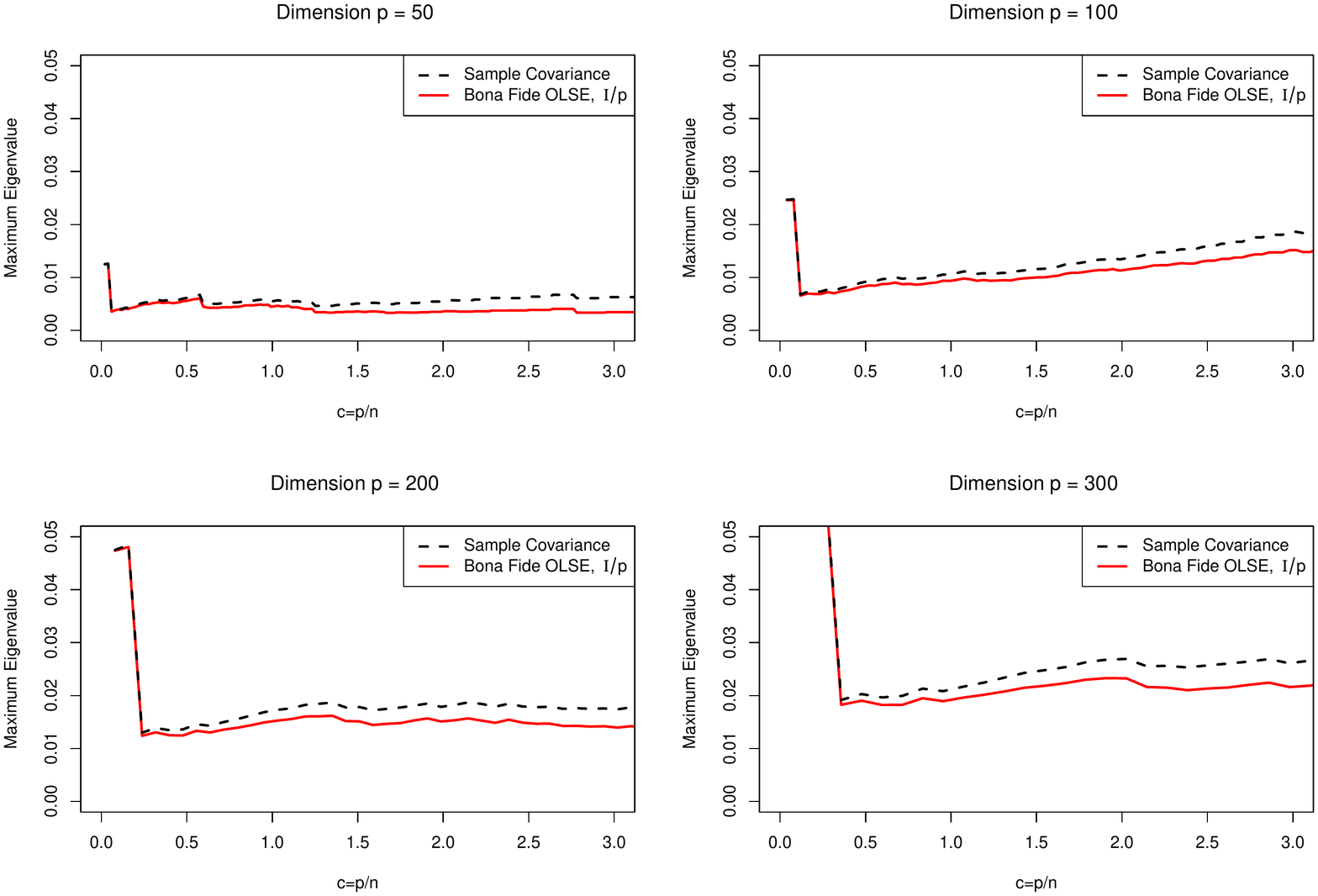}
\caption{Maximum eigenvalue of the OLSE and the sample covariance matrix.}
\label{Fig:6}
\end{figure}

\begin{figure}[H]
\includegraphics[scale=0.5]{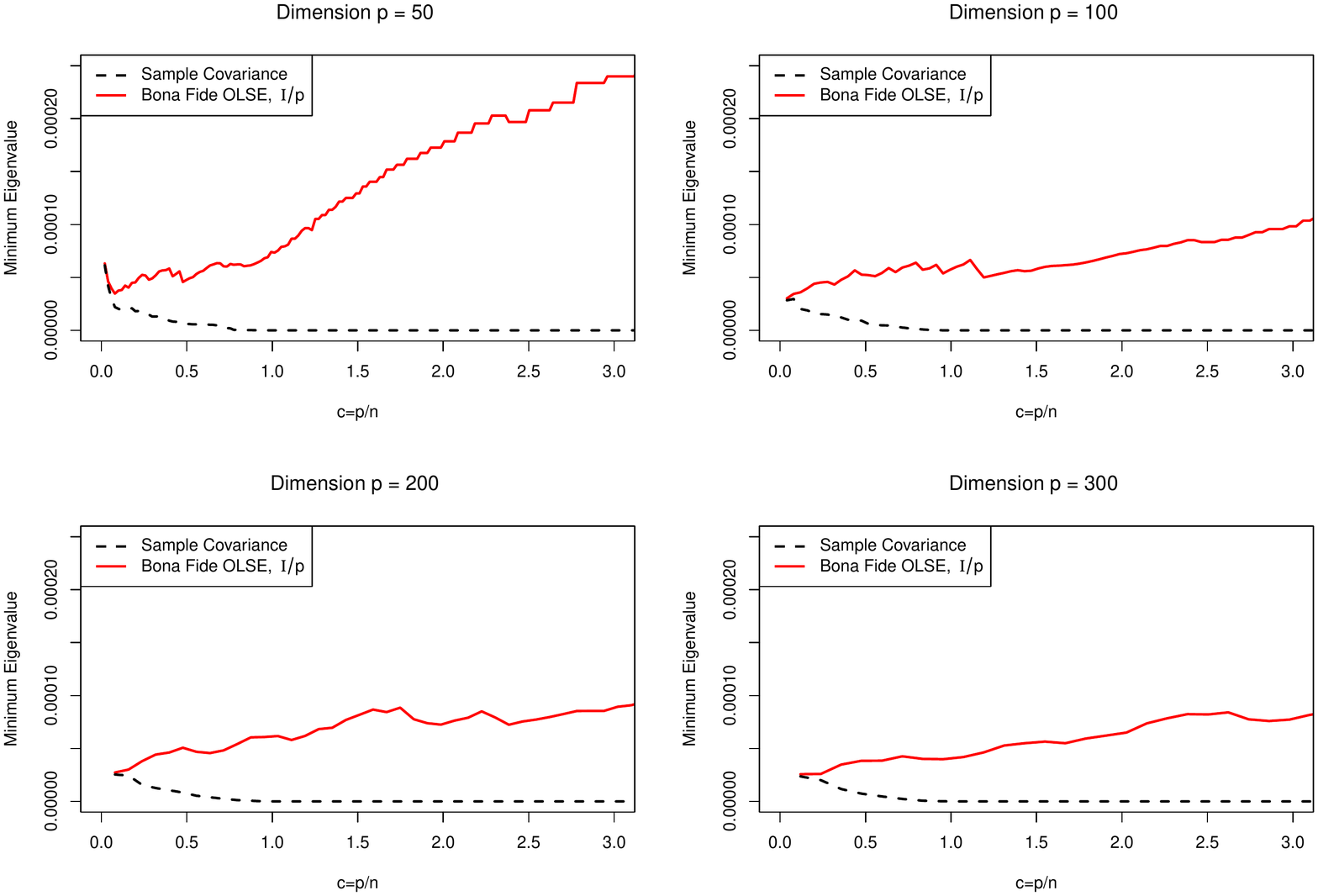}
\caption{Minimum eigenvalue of the OLSE and the sample covariance matrix.}
\label{Fig:7}
\end{figure}

\section{Summary} This paper considers the problem of estimation of the covariance matrix in large dimensions.  This case is known in literature as the large dimensional asymptotics and it includes the number of variables $p\rightarrow\infty$ and the sample size $n\rightarrow\infty$ so that $p/n\rightarrow c\in (0, +\infty)$. Here, we construct the optimal linear shrinkage estimator (OLSE) for the covariance matrix which is proven to have \textit{almost surely} the smallest Frobenius loss asymptotically. It is compared with the linear shrinkage estimator constructed by Ledoit and Wolf (2004). A significant improvement is obtained when some additional \textit{prior} information on the population covariance matrix is available.

\section{Appendix} Here the proofs of the theorems are given.\\

\textbf{Proof of Theorem 3.1}.
\begin{proof}
In order to prove Theorem 3.1 we use the MP equation (\ref{MP}). Before we proceed, we rewrite $\dfrac{1}{p}\text{tr}(\bS_n^2)$ for all $z\in\mathbbm{C}^+$ in the following way
\begin{equation}\label{rew2}
 \dfrac{1}{p}\text{tr}(\bS_n^2)=-\left.\dfrac{1}{2}\dfrac{\partial^2}{\partial z^2}\dfrac{1}{p}\dfrac{\text{tr}(\bS_n-1/z\bI)^{-1}}{z}\right|_{z=0}=-\left.\dfrac{1}{2}\dfrac{\partial^2}{\partial z^2}\left(\dfrac{m_{F_n}(1/z)}{z}\right)\right|_{z=0}\,,
\end{equation}
where $m_{F_n}(1/z)$ is the Stieltjes transform introduced in (\ref{Strans}). From Theorem 2.1 we know that $m_{F_n}(1/z)$ tends almost surely to a nonrandom $m_{F}(1/z)$ which satisfies the MP equation (\ref{MP}).  Using this fact and denoting $\Theta_n(z)=\dfrac{m_{F_n}(1/z)}{z}$ we get that $\Theta_n(z)$ tends almost surely to $\Theta(z)=\dfrac{m_{F}(1/z)}{z}$ which is the unique solution of the following equation
\begin{eqnarray}\label{tetaMP}
 \Theta(z)&=&\int\limits_{-\infty}^{+\infty}\dfrac{1}{z\tau((1-c)-c\Theta(z))-1}dH(\tau)\,.
\end{eqnarray}
 Before we proceed with the second derivative of $\Theta(z)$ with respect to $z$, it is shown that the quantities $\Theta(z)$, $\Theta^\prime(z)$ and $\Theta''(z)$, which appear in the calculation, are bounded at zero.

First, we point out that the limit under the integral sign in (\ref{tetaMP}) can be safely moved by applying the dominated convergence theorem together with the fact that $H$ is a probability measure and the boundedness of $\Theta(z)$. More precisely, let $\underline{m}(z)=-z(1-c)+cm_F(1/z)$ and rewrite (\ref{tetaMP}) in the following way
\begin{equation}\label{tetaMP2}
 \Theta(z)=-\int\limits_{-\infty}^{+\infty}\dfrac{1}{\tau \underline{m}(z)+1}dH(\tau)\,,
\end{equation}
where the function $\underline{m}(z)$ is another Stieltjes transformation of a positive measure on $\mathbbm{R}^+$ (see, Silverstein (1995)). Thus, using inequality (see, Rubio and Mestre (2011, Lemma 6))
\begin{equation}\label{rb}
\left|\dfrac{\tau}{\tau\underline{m}(z)+1}\right|\leq \tau\dfrac{|z|}{\textbf{Im}(z)}\,
\end{equation}
we get
\begin{equation}\label{posmeas}
 \left|\dfrac{1}{\tau \underline{m}(z)+1}\right|\leq\dfrac{|z|}{\textbf{Im}(z)}\,.
\end{equation}
Now, without loss of generality we construct the complex sequence $z_k=ih_k$ such that $h_k\rightarrow0^+$ as $k\rightarrow\infty$. Note that for our framework the real part of $z$ is not essential that is why without loss of generality we put it equal to zero. It follows that
\begin{equation}\label{posmeas2}
 \left|\dfrac{1}{\tau \underline{m}(z_k)+1}\right|\leq\dfrac{|z_k|}{\textbf{Im}(z_k)}=\dfrac{h_k}{h_k}=1\,.
\end{equation}
 So, using the fact that the function under the integral in (\ref{tetaMP2}) is bounded by an integrable function under the probability measure $H$ and the dominated convergence theorem we can move the limit $z\rightarrow0^+$ under the integral sign in (\ref{tetaMP}).
It leads to
\begin{equation}\label{t0}
 \Theta(0)\equiv\lim\limits_{z\rightarrow0^+}\Theta(z)=-1,~~\text{if}~\lim\limits_{z\rightarrow0^+}z\Theta(z)<\infty\,.
\end{equation}
The last inequality is true since if $\lim\limits_{z\rightarrow0^+}z\Theta(z)=\infty$ then from (\ref{tetaMP}) we get that $\lim\limits_{z\rightarrow0^+}\Theta(z)=0$. This leads to $\lim\limits_{z\rightarrow0^+}z\Theta(z)=0$ which contradicts the statement $\lim\limits_{z\rightarrow0^+}z\Theta(z)=\infty$.

A similar analysis is performed for $\Theta^\prime(z)$ given by
\begin{equation}\label{tder}
 \Theta^\prime(z)=-\int\limits_{-\infty}^{+\infty}\dfrac{\tau((1-c)-c\Theta(z)-cz\Theta^\prime(z))}{(z\tau((1-c)-c\Theta(z))-1)^2}dH(\tau)\,.
\end{equation}
Rearranging terms in (\ref{tder}) we get
\begin{eqnarray}\label{tder0}
&\Theta^\prime(z)&\left(1-cz\int\limits_{-\infty}^{+\infty}\dfrac{\tau}{(z\tau((1-c)-c\Theta(z))-1)^2} dH(\tau)\right)\\
&=&\biggl(-(1-c)+c\Theta(z)\biggr)\int\limits_{-\infty}^{+\infty}\dfrac{\tau}{(z\tau((1-c)-c\Theta(z))-1)^2} dH(\tau)\nonumber\,.
\end{eqnarray}
Due to $\Theta(0)<\infty$, we observe that the right hand side of (\ref{tder0}) is bounded at zero and thus the left hand side tends to $\Theta^\prime(0)$ as $z\rightarrow0^+$. It leads to
 \begin{equation}\label{t1}
   \Theta^\prime(0)\equiv\lim\limits_{z\rightarrow0^+}\Theta^\prime(z)=-\int\limits_{-\infty}^{+\infty}\tau dH(\tau)\,.
 \end{equation}
 Next, we derive $\Theta''(z)$. Let
\begin{align}\label{M}
&&M(z,\tau)=-\dfrac{1}{(z\tau((1-c)-c\Theta(z))-1)^2}\\
&&N(z,\tau)=\tau((1-c)-c\Theta(z)-cz\Theta^\prime(z))\label{N}\,.
 \end{align}
 Using (\ref{tder}) together with (\ref{M}) and (\ref{N}), $\Theta''(z)$ can be rewritten in the following way
 \begin{equation}\label{tder2}
 \Theta''(z)=\int\limits_{-\infty}^{+\infty}N^\prime(z,\tau)M(z,\tau)dH(\tau)+\int\limits_{-\infty}^{+\infty}M^\prime(z,\tau)N(z,\tau)dH(\tau)\,,
\end{equation}
where
 \begin{eqnarray}\label{M1}
&&M^\prime(z,\tau)=-2M(z,\tau)\dfrac{N(z,\tau)}{z\tau((1-c)-c\Theta(z))-1)}\\
&&N^\prime(z,\tau)=\tau\bigl(-2c\Theta^\prime(z)-cz\Theta''(z)\bigr)\label{N1}\,.
 \end{eqnarray}
 From (\ref{tder2}) we get
\begin{eqnarray}\label{tder22}
&\Theta''(z)&\left(1-cz\int\limits_{-\infty}^{+\infty}\dfrac{\tau}{(z\tau((1-c)-c\Theta(z))-1)^2} dH(\tau)\right)\\
&=&2c\Theta^\prime(z)\int\limits_{-\infty}^{+\infty}\dfrac{\tau}{(z\tau((1-c)-c\Theta(z))-1)^2} dH(\tau)\nonumber\\
&+&\int\limits_{-\infty}^{+\infty}M^\prime(z,\tau)N(z,\tau)dH(\tau)\nonumber\,.
\end{eqnarray}

 Taking the limit as $z\rightarrow0^+$ in (\ref{M}), (\ref{N}) and (\ref{M1}) as well as using (\ref{t1}) and (\ref{t0}) we obtain
  \begin{eqnarray}
  && \lim\limits_{z\rightarrow0^+}M(z,\tau)=-1,\label{M21}\\
  && \lim\limits_{z\rightarrow0^+}N(z,\tau)=\tau,\label{N21}\\
&& \lim\limits_{z\rightarrow0^+}M^\prime(z,\tau)=-2\tau\label{M2}\,.
 \end{eqnarray}
Hence, from (\ref{tder22}) together with (\ref{M21}), (\ref{N21}), (\ref{M2}), (\ref{t1}) and (\ref{t0}) we get that $\Theta''(0)<\infty$ and
 \begin{equation}\label{tder21}
\Theta''(0)\equiv \lim\limits_{z\rightarrow0^+}\Theta''(z)=-2c\left(\int\limits_{-\infty}^{+\infty}\tau dH(\tau)\right)^2-2\int\limits_{-\infty}^{+\infty}\tau^2 dH(\tau)\,.
\end{equation}
Now the result of Theorem 3.1 follows from (\ref{rew2}), (\ref{tetaMP}) and (\ref{tder21}).\\
\end{proof}

\textbf{Proof of Theorem 3.2}.
\begin{proof}
In order to prove the statement of the Theorem 3.2 we use the following lemma of Rubio and Mestre (2011).
\begin{lemma}\textbf{[Lemma 4, Rubio and Mestre (2011)]}
Let $\{\bol{\xi}_1,\ldots,\bol{\xi}_n\}$ be a sequence of i.i.d. real random vectors with zero mean vector, identity covariance matrix, and uniformly bounded $4+\varepsilon$ moments for some $\varepsilon>0$ and let $\bC_n$ be some nonrandom matrix (possibly random but independent of $\bol{\xi}_n$) with bounded trace norm at infinity. Then
\begin{equation}\label{lemma}
\left|\dfrac{1}{n}\sum\limits_{i=1}^n\bol{\xi}_i^\prime\bC_n\bol{\xi}_i-\text{tr}(\bC_n) \right|\longrightarrow0~\text{a. s.}~~\text{as}~n\rightarrow\infty\,.
\end{equation}
\end{lemma}
Next, we proceed to the proof of Theorem 3.2 directly by considering the asymptotic behavior of the following two quantities
\begin{eqnarray}
&&\eta_1=\dfrac{1}{p}\text{tr}(\bS_n\bTheta)=\dfrac{1}{n}\sum\limits_{i=1}^n\mathbf{x}_i^\prime\left(\dfrac{1}{p}\bTheta^{1/2}\bSigma_n\bTheta^{1/2}\right)\mathbf{x}_i\label{eta1}\\
&&\eta_2=\dfrac{1}{p}||\bS_n||_F^2=\dfrac{1}{p}\text{tr}(\bS_n^2)\label{eta2}\,,
\end{eqnarray}
where $\mathbf{x}_i$ is the $i$th column of the matrix $\bx_n$ defined in (\ref{obs}).

First, we prove that for $\dfrac{p}{n}\rightarrow c\in (0, +\infty)$ as $n\rightarrow\infty$ the following assertion holds
\begin{equation}\label{eta_1}
\Biggl|\eta_1-\dfrac{1}{p}\text{tr}(\bSigma_n\bTheta)\Biggr|\longrightarrow 0~\text{a. s.}~~\text{for}~~n\rightarrow\infty\,.
\end{equation}
For the application of Lemma 7.1 we have to show that the trace norm of the matrix $\dfrac{1}{p}\bTheta^{1/2}\bSigma_n\bTheta^{1/2}$ is bounded. It holds that (see, Fang et al. (1994))
{\small
\begin{equation}\label{Rnorm2}
\left|\left|\dfrac{1}{p}\bTheta^{1/2}\bSigma_n\bTheta^{1/2}\right|\right|_{tr}=\dfrac{1}{p}\text{tr}(\bTheta^{1/2}\bSigma_n\bTheta^{1/2})=\dfrac{1}{p}\text{tr}(\bTheta\bSigma_n) \leq\dfrac{\text{tr}(\bTheta)\tau_{max}(\bSigma_n)}{p}\,,
\end{equation}
}
where $\tau_{max}(\bSigma_n)$ denotes the largest eigenvalue of the matrix $\bSigma_n$. At last, using the assumption (A5) we deduce the boundedness of the trace norm of the matrix $\dfrac{1}{p}\bTheta^{1/2}\bSigma_n\bTheta^{1/2}$.
Thus, using Lemma 7.1 we get the statement (\ref{eta_1}).

Next, we prove the main result of Theorem 3.2, namely the following statement
\begin{equation}\label{eta_2}
\Biggl|\eta_2- \dfrac{1}{p}\biggl(||\bSigma_n||^2_F+\dfrac{c}{p}||\bSigma_n||^2_{tr}\biggr)\Biggr|\longrightarrow0~\text{a. s.}~\text{for}~\dfrac{p}{n}\rightarrow c\in(0, +\infty)~\text{as}~n\rightarrow\infty\,.
\end{equation}
In order to prove (\ref{eta_2}) we use the result of Theorem 3.1. First, we rewrite the difference in (\ref{eta_2}) via the triangle inequality in the following way
\begin{equation}\label{1step}
\Biggl|\eta_2-\dfrac{1}{p}\biggl(||\bSigma_n||^2_F+\dfrac{c}{p}||\bSigma_n||^2_{tr}\biggr)\Biggr|\leq\Biggl|\eta_2-\phi\Biggr|+\Biggl|\phi-\dfrac{1}{p}\biggl(||\bSigma_n||^2_F+\dfrac{c}{p}||\bSigma_n||^2_{tr}\biggr)\Biggr|\,,
\end{equation}
where $\phi=\int\limits_{-\infty}^{+\infty}\tau^2dH(\tau)+c\left(\int\limits_{-\infty}^{+\infty}\tau dH(\tau)\right)^2$ is given in Theorem 3.1. Next we show that the right hand side of (\ref{1step}) vanishes almost surely as $n\rightarrow\infty$. Using Theorem 3.1 we get immediately that for the first term in (\ref{1step}) holds
\begin{equation}\label{1term}
\Biggl|\eta_2-\phi\Biggr|\longrightarrow0~\text{a. s.}~\text{for}~n\rightarrow\infty
\end{equation}
The second term in (\ref{1step}) is nonrandom. Next, we show that it approaches to zero as $n\rightarrow\infty$.
Due to assumption (A3) it holds that $H_n(t)$ tends to $H(t)$ at all continuity points of $H(t)$. Thus,
\begin{equation}\label{tr1}
\dfrac{1}{p}||\bSigma_n||^2_F=\dfrac{1}{p}\text{tr}(\bSigma^2_n)=\dfrac{1}{p}\sum\limits_{i=1}^p\tau^2_i=\int\limits_{-\infty}^{+\infty}\tau^2 dH_n(\tau)\stackrel{{n\rightarrow\infty}}{\longrightarrow}\int\limits_{-\infty}^{+\infty}\tau^2 dH(\tau)\,.
\end{equation}
The integral in (\ref{tr1}) exists due to assumption (A5).
Similarly, it can be shown that
\begin{equation}\label{tr2}
\dfrac{1}{p^2}||\bSigma_n||^2_{tr}\stackrel{{n\rightarrow\infty}}{\longrightarrow}\left(\int\limits_{-\infty}^{+\infty}\tau dH(\tau)\right)^2\,.
\end{equation}
From (\ref{tr1}) and (\ref{tr2}) it follows that
\begin{equation}\label{2term}
\Biggl|\phi-\dfrac{1}{p}\biggl(||\bSigma_n||^2_F+\dfrac{c}{p}||\bSigma_n||^2_{tr}\biggr)\Biggr|\longrightarrow0~~\text{for}~n\rightarrow\infty\,.
\end{equation}
As a result, (\ref{1term}) and (\ref{2term}) complete the proof of Theorem 3.2.\\
\end{proof}

\section*{Acknowledgements}

\noindent The authors are grateful to the referee and the editor for their suggestions, which have improved the presentation in the paper. The first author is partly supported by the German Science Foundation (DFG) via the Research Unit 1735 ''Structural Inference in Statistics: Adaptation and Efficiency''. He also appreciates the financial support of the German Science Foundation (DFG) via the projects BO3521/2-2 and OK103/1-2 ''Wishart Processes in Statistics and Econometrics: Theory and Applications''.

\bibliographystyle{elsarticle-num}

\end{document}